\documentclass[11pt,centertags]{amsart}
\usepackage{amssymb}            % extra common ams symbols -- not autoloaded
\usepackage{mathrsfs}             % Adds Script Fonts for Curly scripts see \scr below
\usepackage{hyperref}           % click to jump to refernce in .dvi
\usepackage{ifsym}              % lots of symbols, weather, clocks, legends, random
\usepackage{calc}               % allows arithmetic calculation with quantities
\usepackage{enumerate}          % improved enumeration environment control
\usepackage[all]{xy}            % all the xy picture goodies
\usepackage[usenames]{color}    % for color text
\usepackage{verbatim}           % verbatim classes
\usepackage{fancyhdr}           % super headers --recommended
\usepackage{url}                % appropriately display url's
\usepackage{gloss}              % some nice formatting specials
\usepackage{yhmath}             % yh math fonts

\newif\ifPDF
\ifx\pdfoutput\undefined
    \PDFfalse
\else
    \ifnum\pdfoutput > 0
        \PDFtrue
    \else
        \PDFfalse
    \fi
\fi
\ifPDF
    \usepackage[pdftex]{graphicx}
    \DeclareGraphicsExtensions{.pdf,.png,.jpg}
    \graphicspath{{}}
\else
    \usepackage{graphicx}
    \DeclareGraphicsExtensions{.eps}
    \graphicspath{{}}
\fi

\newtheorem*{main*}{Main Theorem}
\newtheorem{main}{Theorem}

\newtheorem{theorem}{Theorem}[section]
\newtheorem*{theorem*}{Theorem}
\newtheorem{proposition}[theorem]{Proposition}

\newtheorem{corollary}[theorem]{Corollary}

\newtheorem*{question*}{Question}

\newtheorem*{conjecture*}{Conjecture}

\theoremstyle{definition}

\newtheorem*{definition*}{Definition}

\newtheorem{examples}[theorem]{Examples}

\theoremstyle{remark}
\newtheorem{remark}[theorem]{Remark}

\numberwithin{equation}{section}

\newcommand{\R}{\mathbb{R}}
\newcommand{\RH}{\mathbb{H}}

\newcommand{\mc}{\mathcal}
\newcommand{\mbf}{\mathbf}

\newcommand{\mf}{\mathfrak}

\newcommand{\scr}{\mathscr}

\newcommand{\Ga}{\Gamma}
\newcommand{\ga}{\gamma}
\newcommand{\eps}{\varepsilon }
\newcommand{\La}{\Lambda}

\DeclareMathOperator{\Iso}{Iso}

\DeclareMathOperator{\vol}{vol}

\DeclareMathOperator{\SL}{SL}

\DeclareMathOperator{\PSL}{PSL}

\DeclareMathOperator{\Sp}{\Sp}

\newcommand{\op}[1]{\operatorname{#1}}
\newcommand{\set}[1]{\left\{#1 \right\}}

\newcommand{\directsum}{\oplus}

\newcommand{\dvol}{\operatorname{dvol}}

\newcommand{\pa}{\partial }
\newcommand{\bd}{\partial_\infty}

\newcommand{\goto}[1]{\stackrel{#1}{\longrightarrow}}

\providecommand{\to}{\longrightarrow }

\newcommand{\abs}[1]{\left\lvert #1 \right\rvert }
\newcommand{\norm}[1]{\left\| #1 \right\| }
\newcommand{\inner}[1]{\left\langle #1 \right\rangle }

\newcommand{\bs}{\backslash}

\def\[#1\]{\begin{align*}\begin{split} #1 \end{split}\end{align*} }
\def\[[#1\]]{\begin{align}\begin{split} #1 \end{split}\end{align} }

\renewcommand{\tilde}{\widetilde}
\renewcommand{\hat}{\widehat}

\renewcommand{\(}{\left(}
\renewcommand{\)}{\right)}
\newcommand{\Fc}{\mathcal{F}}

\newcommand{\Lam}{\scr{L}}
\DeclareMathOperator{\rank}{rank}
\newcommand{\Khat}{\hat}
\renewcommand{\L}{L}
\newcommand{\kk}{\mc{K}}

\begin{document}

\address{Indiana University}
%\email{cconnell@math.indiana.edu}

\address{Universidad de la Rep\'{u}blica}
%\email{martinez.mte@}

\author[C. Connell]{Chris Connell$^\dagger$}
\thanks{$\dagger$ The author was
  supported by NSF grant DMS-0608643.}
\author[M. Mart{\'i}nez]{Matilde Mart{\'i}nez$^\ddagger$}
\thanks{$\ddagger$ The author was supported by ANII, research grant FCE2007\underline{ }106.}

\title[Harmonic Measures]{Harmonic and Invariant Measures on Foliated Spaces}

\begin{abstract}
We consider the family of harmonic measures on a lamination $\Lam$ of a compact space $X$
by locally symmetric spaces $\L$ of noncompact type, i.e. $\L\cong \Gamma_\L\bs G/K$. We
establish a natural bijection between these measures and the measures on an associated
lamination foliated by $G$-orbits, $\hat\Lam$, which are right invariant under a minimal
parabolic (Borel) subgroup $B<G$. In the special case when $G$ is split, these measures
correspond to the measures that are invariant under both the Weyl chamber flow and the
stable horospherical flows on a certain bundle over the associated Weyl chamber
lamination. We also show that the measures on $\hat\Lam$ right invariant under two
distinct minimal parabolics, and therefore all of $G$, are in bijective correspondence
with the holonomy invariant ones.
\end{abstract}
\maketitle

\thispagestyle{empty}

%\selectlanguage{french}
%
%\maketitle
%
%\begin{abstract} $\Lam$ denotes a (compact, non-singular) lamination by hyperbolic Riemann surfaces.
%We prove that a probability measure on $\Lam$ is harmonic if and only if it is the
%projection of a measure on the unit tangent bundle $\hat\Lam$ of $\Lam$ which is invariant
%under both the geodesic and the horocycle flows.
%\end{abstract}
%\begin{otherlanguage}{french}
%\begin{abstract}
%$\Lam$ denote une lamination (compacte, non singuli\`ere) par surfaces de Riemann
%hyperboliques. On montre qu' une mesure sur $\Lam$ est harmonique si et seulement si
%elle est la projection d'une mesure sur le fibr\'e tangent unitaire $\hat\Lam$ qui est
%invariante sous les flots g\'eodesique et horocyclique.
%\end{abstract}
%\end{otherlanguage}
%
%
%

%%%%%%%%%%%%%%%%%%%%%%%%%%%%%%%%%%%%%%%%%%
%%%%%%%%%%%%%%%%%%%%%%%%%%%%%%%%%%%%%%%%%%
%%%%%

\section*{Introduction}

This article explores measures associated to a nonsingular lamination of an arbitrary
compact topological space by locally symmetric spaces. (Note that some authors use the
term ``foliated space'' to mean lamination.) Throughout the paper,~$\Lam$ will denote a
lamination of a topological space $X$ by leaves that are discrete quotients of the same
symmetric space, $G/K$, of noncompact type and which is equipped with transversely
continuously varying smooth metrics. In particular, each leaf of $\Lam$ carries a locally
symmetric metric of nonpositive sectional curvature. We do not assume that these metrics
are uniformized in any particular way. Hence the curvatures, for instance, can vary from
leaf to leaf. Throughout we will assume that the space $X$ supporting $\Lam$ is compact.

We consider two different families of measures associated to~$\Lam$. On the one hand, we
have the measures invariant under the heat diffusion along the leaves of $\Lam$ which are
called \emph{harmonic measures}. These measures often arise in the study of topological
and ergodic properties of a foliation, in part because they generalize the holonomy
groupoid invariant measures which may not always exist. On the other hand, we have the
class of measures on the natural $K$-bundle lamination ~$\hat\Lam$ associated to $\Lam$
that are invariant under the right action by a fixed common choice of Borel subgroup of
$B<G$. In the split case, i.e. when the intersection of every Cartan subgroup with $K$ is
trivial, this latter group of measures has a dynamical characterization. They are the
measures which are invariant under the combined Weyl chamber flow and horospherical
(unipotent) flows of Weyl chambers.

In the case when $\rank(G)\geq 2$, there are two competing notions of harmonicity: vanishing under the standard differential geometric laplacian versus the stricter requirement of vanishing under all second order elliptic $G$-invariant operators which annihilate constants. However, we will show (see Section \ref{subsec:harmonic}) that in our setting these notions coincide.

%correct the decomposition -ok
Our main result reads as follows:

\begin{main}\label{thm:main}
The canonical projection $\hat\Lam\to\Lam$ induces a bijective correspondence between the
$B$-invariant measures on $\hat\Lam$ and the $\Delta_\Lam$-harmonic measures on $\Lam$.
\end{main}

When $\Lam$ is at least transversally continuous, $B$-invariant
measures always exist on $\hat\Lam$ since $B$ is an amenable
group acting continuously on the compact space $\hat\Lam$. In
particular, this gives another proof of the existence of
harmonic measures. Theorem \ref{thm:main} also holds in the
case where $\Lam$ is only transversely measurable, but then the
property of being a harmonic measure becomes more restrictive,
and there may not exist any (see Section
\ref{subsec:harmonic}).% do we even want this comment?

Naturally, the fiber over a given harmonic measure under the induced canonical
projection of measures consists of more than just the unique $B$-invariant measure. We
will also describe measures in this fiber which are invariant only under the right action
by the subgroup $N\rtimes A$ of $B=NAM$. (See Section \ref{sec:prelims} for
definitions.)

%Naturally, the theorem remains true if we exchange the \emph{stable} unipotent action
%for the \emph{unstable} unipotent action. However, we show that not every harmonic
%measure is invariant under both stable and unstable unipotent actions.  In fact, we show
%that there is a canonical bijection between measures invariant under both of these
%unipotent actions and the simplest kind of harmonic measures; namely, those coming
%from measures on transversals which are invariant under the holonomy pseudogroup.

\begin{remark}
The main result of this article was carried out in the special case of surface
laminations, namely when $G=PSL(2,\R)$, by the second author and Yu.Bakhtin
in the two articles \cite{Martinez:06} and \cite{Bakhtin-Martinez:08}. In this
paper, besides establishing the results in complete generality, we also manage
to significantly shorten the  proofs, albeit at the expense of a bit more
machinery.
\end{remark}

We also establish a couple of other related results along the way. Given a
harmonic measure $\mu$, in Theorem \ref{thm:same} we give an explicit
geometric construction for the unique $B$-invariant lift $\hat\mu$ on $\hat
\Lam$ of any given harmonic measure $\mu$ on $\Lam$.

Finally in the last section, we will provide a short proof of the following related result.
\begin{main}[\cite{Muniz-Manasliski:09}]\label{thm:G-inv}
There is a natural bijective correspondence between $G$-invariant probability measures on
$\hat\Lam$ and invariant transverse measures on $\Lam$.
\end{main}

In Section \ref{sec:prelims} we review those concepts that we will need from
the theory of semisimple Lie groups and locally symmetric spaces. In Section
\ref{sec:laminations} we review laminations by symmetric spaces and
harmonic measures. Section \ref{sec:proof} is devoted to the proof of Theorem
\ref{thm:main}. In Section \ref{sec:lift} we give an explicit construction of a
measure on $\hat\Lam$ invariant under a minimal parabolic subgroup that
projects onto a given harmonic measure on~$\Lam$. Finally, in Section
\ref{sec:examps} we give an interesting application along with the proof
Theorem \ref{thm:G-inv}.

 %This will be reduced to a probabilistic
%statement on the regularity of the distribution of the radial component of the Brownian
%motion on the universal symmetric space in Section 3. The uniqueness of the measure
%will be shown in Section 4 and in Section 5 we establish the result about transverse
%invariant measures. Finally in Section 5 we construct several interesting examples of
%such laminations and applications to harmonic measures and uniquely ergodic
%measures.

\section*{Acknowledgments} Both authors are grateful for the hospitality of Andr\'{e}s
Navas at the Universidad de Santiago de Chile where this work was initially conceived. We
are also thankful to the anonymous referee for a careful reading and numerous corrections
and helpful suggestions. The second author is thankful for helpful comments from
M.~Brunella communicated by C.~Bonatti.

%%%%%%%%%%%%%%%%%%%%%%%%%%%%%%%%%%%%%%%%%%
%%%%%%%%%%%%%%%%%%%%%%%%%%%%%%%%%%%%%%%%%%
%%%%%%
\section{The action of $G$ and the Weyl chamber flow and unipotent flows}\label{sec:prelims}

Let $G$ be a semisimple Lie group of noncompact type with Lie algebra
$\mf{g}$. Choose a Cartan subalgebra $\mf{h}\subset\mf{g}$. Let $\La$ be
the corresponding set of  nonzero roots giving the corresponding root space
decomposition $\mf{g}=\mf{h}\directsum_{\alpha\in \La}\mf{g}_\alpha$,
where $\mf{h}$ is identified with the $0$-root space $\mf{g}_0$. Moreover we
have the associated Cartan decomposition $\mf{g}=\mf{k}+\mf{p}$.   We
denote the Cartan subspace of $\mf{p}$ by $\mf{a}=\mf{p}\cap \mf{h}$ and
its centralizer by $\mf{m}=\mf{k}\cap \mf{h}$.

Let $K$ be the maximal compact subgroup of $G$ whose Lie algebra is
$\mf{k}$. The projection $G\to G/K$ is a diffeomorphism when restricted to
the subset $\exp(\mf{p})\subset G$. Moreover under the identification of $G$
with the group of isometries $\Iso(G/K)$, $K$ can be identified with the
isotropy subgroup $K_p$ of some point $p\in G/K$. Doing so, we can describe
any geodesic through $p$ as $\exp(t h)\cdot p$ where $h\in\mf{p}$.
Moreover, $h$ belongs to at least one maximal abelian subalgebra living in
$\mf{p}$. By varying $p$, and hence the choice of maximal compact $K=K_p$
and the corresponding Lie group and Lie algebra decompositions, all geodesics
can be described as orbits under left actions by one parameter subgroups.

We let $A=\exp(\mf{a})$, and denote the centralizer and normalizer of $A$ in
$K$ by $M$ and $M'$ respectively. They both have the same Lie algebra
$\mf{m}$ introduced above. Moreover,  $M$ is normal in $M'$ and $W=M'/M$
is the corresponding Weyl group.
%Setting $\mf{n}_0$ to be the orthogonal complement of
%$\mf{a}\directsum\mf{k}_0$ in $\mf{h}$,
Choosing a positive Weyl chamber $\mf{a}^+\subset \mf{a}$, we obtain a
corresponding set of positive roots $\La^+$. We write the maximal nilpotent
subalgebra of $\mf{g}$ corresponding to this decomposition as
$\mf{n}=\directsum_{\alpha\in\La^+}\mf{g}_\alpha$. The corresponding
maximal unipotent subgroup will be denoted by $N<G$. The corresponding
{\em Borel} subgroup is the maximal solvable subgroup $B=MAN<G$ which
corresponds to the maximal solvable subalgebra
$\mf{b}=\mf{h}\directsum\mf{n}=\mf{m}\directsum\mf{a}\directsum\mf{n}$.
%This lives in the same minimal parabolic subgroup $B=NAM$ whose Lie algebra
%is $\mf{h}\directsum_{\alpha\in\La^+}\mf{g}_\alpha$, not to be confused with
%the subspace $\mf{p}$.

Recall that a semisimple Lie group has the following Bruhat decomposition $G=\cup_{w\in
W}BwB$ where $W<K$ is the isomorphic copy of the Weyl group generated by reflections across
chamber walls of $\mf{a}$. Since the Weyl group is simply transitive on chambers, $M$ can
also be described as the stabilizer of any single Weyl chamber of $\mf{a}$. Therefore we can
identify $G/M$ as the bundle of Weyl chambers of the symmetric space $G/K$. In the case that
$G$ has rank $1$, $G/M$ is just the (total space of the) unit tangent bundle. Note that by
the Iwasawa decomposition $G=KAN=NAK$, we may identify $G/K$ with a solvable subgroup $NA$,
and the Weyl chamber bundle $G/M$ with $NA(K/M)$. For instance if $G=PO(n,1)=\Iso(\RH^n)$
then $K/M=SO(n)/SO(n-1)\cong S^{n-1}$ so $G/M= NAK/M\cong \RH^n \times S^{n-1}$ is naturally
identified with the unit tangent bundle $T^1\RH^n$.

Note that the right translation action of $M$ on both $G/M$ and $G/K$ is
trivial. For $G$ of arbitrary rank, the geodesic flow naturally generalizes to the
Weyl chamber flow by the right action of $A$ on $G/M$. (Of course the
geodesic flow on $T^1(G/K)$ exists in higher rank as well, but it is not induced
by a right action.) Since $M$ is the centralizer of $A$ this action makes sense
and commutes with the quotient of the action of $M$ on $G$. In other words,
for any $a\in A$, we have $gMa=gaM$.

This action of $A$ on $G/M$ is simply the quotient under $M$ of the right
action of $A$ on $G$. Dissimilarly, the right action of $N$ on $G$ does not
descend to an action on $G/M$ since $M$ does not, in general, centralize $N$,
and hence the action of $N$ does not commute with $M$. More explicitly, the
map $gM\mapsto gnM$ is not well defined since for any $k\in M$ we have
$gM=gkM$, but $gnM\neq gknM$.

We will need to introduce a few other notions. The Furstenberg boundary of
$G/K$ is the space $G/B$. From the Iwasawa decomposition, we naturally see
that $G/B=K/M$ is the compact space of Weyl chambers of $G$ based at $e\in
G$. Since $N<B$, the left action by this subgroup fixes points in $G/B$. Every
Weyl chamber $C\subset G/K$ asymptotically approaches a Weyl chamber $C'$
which passes through $eK$. More specifically $\bd C=\bd C'\subset \bd G/K$.
Hence the Weyl chamber $C$ at infinity can be identified with some point
$[C]=[C']=q\in G/B$.  The action of $N$ on $G/B$ fixes the canonical point
$eB\in G/B$ and each element of $N$ fixes no other point. Moreover, for any
element $q\in G/B$ there is an element $k\in K$, unique modulo $M$, such that
$kq=eB\in G/B$.

%\section{Harmonic measures as projections of measures invariant under the affine group.}

\section{Laminations by symmetric spaces}\label{sec:laminations}

We say $\Lam $ is a $C^{r,s,u}$ \emph{lamination} for
$r,s,u\in[0,\infty]\cup\set{\omega}$ with $s\leq r$ if there is a separable, locally compact space
$X$ that has an open covering $\{ E_i\}$ and an atlas $\{(E_i,\varphi_i)\}$
satisfying:

\begin{enumerate}
\item $\varphi_i:E_i\rightarrow D_i\times T_i$ is a homeomorphism, for some open
    ball $D_i$ in $\R^n$ and a topological space $T_i$, and

\item the coordinate changes $\varphi_j\circ\varphi_i^{-1}$ are  of the form
    $(z,t)\mapsto (\zeta (z,t),\tau(t))$, where each $\zeta$ is $C^r$ smooth in the $z$
    variable with all derivatives up to order $s$ varying in a $C^u$ smooth way in the
    $t$-variable .
\end{enumerate}

This last condition says that the sets of the form $\varphi_i^{-1}(D_i\times \{t\})$, called
\emph{plaques}, glue together to form $n$-dimensional $C^r$-manifolds that we call
\emph{leaves} which vary in a $C^u$ way up to order $s$. Said differently, the $s$-jet
transversal holonomies are $C^u$.

We will also be interested in considering measurable $C^{r,s}$ laminations, which are $C^r$
smooth in the $z$ variable with all derivatives up to order $s$ varying measurably in the $t$
variable. As a shorthand, a continuous (resp. measurable) lamination means a
$C^{\infty,\infty,0}$ (resp. $C^{\infty,\infty}$ measurable) lamination. In what follows we
will suppress the mention of regularity, as the needed regularity is generally obvious from
the context.

In any class of maps between foliated spaces we can define the subclass of laminated maps to
be those which carry leaves to leaves. A laminated fibration (or fibre bundle) $E\to X$ is a
fibration map between the laminated spaces $E$ and $X$ which is a laminated map and whose
local product structure is compatible with the local product structure of the laminations. As
an example, given  a lamination $\Lam$ by $C^1$ manifolds, we can define the tangent bundle
lamination $T\Lam$ to be the lamination which is locally the product of the tangent space to
the leaves with the same transversals as $\Lam$. Similarly, an assignment of metrics to the
leaves of $\Lam$ corresponds to a section of the laminated bundle $S^2T^*\Lam\to \Lam$ of
symmetric 2-tensors on the cotangent bundle $T^*\Lam$.

We will say that a lamination $\Lam$ is a \emph{locally
symmetric lamination} if the disks $D_i$ are open subsets of
$G/K$ and the maps $\zeta$ are isometries in the $z$ variable.
When the underlying space $X$ is compact, each leaf $L$ of
$\Lam$ will be a complete locally symmetric space, and
therefore each of the form $\Ga_\L\bs G/K$ for some discrete
subgroup $\Ga_\L<G$. Each leaf on a locally symmetric
lamination $\Lam$ will be endowed with the locally symmetric
metric resulting from pulling back by the chart. The transverse
regularity of $\Lam$, being the same as that of the isometries
$\zeta(\cdot,t)$, will be the transverse regularity of the
identification of the leaves with the locally symmetric spaces
$\Ga_L \bs G/K$.

For any locally symmetric lamination $\Lam$, we can construct its associated $K$-bundle
lamination $\Khat \Lam$ as follows. First construct the principal $K$-bundle $\Khat
X\goto{\pi} X$ over $X$ whose fiber over a point $x\in L$, identified with the double
coset $\Ga_L g_x K$ for some $g_x\in G$ is simply again $\Ga_L g_x K$, but now viewed as
a set of left cosets.

The bundle charts $(\Khat E_i,\Khat \varphi_i)$ with $\Khat\varphi_i: E_i\to K\times D_i\times
T_i$ cover the chart $(E_i,\varphi_i)$ from $\Lam$ in the sense that the following diagram
commutes,
$$\xymatrix{
  \Khat{E_i} \ar[d]_{\pi} \ar[rr]^{\Khat{\varphi_i}}
               & & K\times D_i\times T_i \ar[d]^{\pi_0\times \op{id}}  \\
  E_i  \ar[rr]_{\varphi_i}
               & &    D_i\times T_i }$$
where $\pi_0$ is the natural projection and $\pi$ is the restricted bundle map.

The coordinate changes $\Khat\varphi_j\circ{\Khat{\varphi}}_i^{-1}$
are  of the explicit form $(k,z,t)\mapsto (\eta(k,z,t),\zeta
(z,t),\tau(t))$ where $\zeta$ and $\tau$ are as above, and
$\eta$ is an isometry of $K$. Finally we form the lamination
$\Khat\Lam$ of the underlying space $\Khat X$ by taking the lifts
of the leaves $L$ of $\Lam$ under $\pi$ to obtain leaves $\Khat
L$ each of dimension $n+\dim(K)$. Strictly speaking, the charts
$(\Khat E_i,\Khat \varphi_i)$ do not form a set of lamination
charts for $\Khat{\Lam}$ since $K$ is not generally a
topological disk. This can easily be taken care of by finitely
subdividing the charts $(\Khat E_i,\Khat \varphi_i)$ further into
restricted charts $(\Khat {U_\alpha},\Khat{\phi_\alpha})$ where
for each new index $\alpha$, $\Khat{U_\alpha}\subset \Khat{E_i}$
for some $i=i(\alpha)$ and $\Khat{\phi_\alpha}:\Khat{U_\alpha}\to
\Khat{D_\alpha}\times T_i$ is a homeomorphism equal to the
restriction of $\Khat{\varphi_i}$ and where
$\Khat{D_\alpha}\subset K\times D_i$ is a topological ball.

We endow each  leaf $\Khat \L$ of $\Khat\Lam$ with the unique
right $K$-invariant metric which projects to the metric on the
leaf  $\L=\Gamma_\L\bs G/K$ of $\Lam$. Thus each leaf
$\Khat{\L}$ is isometric to $\Gamma_\L\bs G$ and the
regularity of $\Khat\Lam$ will be identical to that of $\Lam$.

What is less obvious is that there is a consistent right action of $G$ on $\Khat{\Lam}$.
In other words, the identification of each leaf $\L$ with $\Gamma_L\bs G/K$ can be done
consistently with the foliation structure. This is the result of Proposition 2.5 of
\cite{Zimmer:88a}.

The transverse regularity of the isometries $\eta$ and $\zeta$
guarantees that the right $G$ action on each leaf of $\Lam$ has
the same transverse regularity. In other words, action can be
defined on each leaf using the identification $\L=\Gamma_\L\bs
G/K$, and the compatibility and regularity of charts guarantee
the regularity, locally and hence globally, of the action. Note
that this does not imply that the $\Ga_\L$ vary continuously,
and indeed they generally do not.

Lastly, we note that we can similarly define the lamination $\check{\Lam}$ whose leaves are
of the form $\L=\Ga_\L\bs G/M$ as bundle over $\Lam$. Since the metric on the leaves of
$\Khat{\Lam}$ are both right and left $K$-invariant, and $M<K$, we could also realize
$\check{\Lam}$ as a leafwise metric quotient of $\Khat{\Lam}$.

%If $\Lam$ is a lamination by locally symmetric leaves $\L=\Gamma_\L\bs
%G/K$ of dimension $n$, then we denote by $\Khat{\Lam}$ the $n+\dim(K)$
%dimensional lamination whose leaves are the canonical $K$ bundles of the
%leaves of $\Lam$ and the flow boxes are canonical lifts of those of $\Lam$.
%Namely, for $U_\alpha$ open in $\Lam$ and $U$ an open ball of $G/K$ such
%that $\phi: U_\alpha\to U\times T_\alpha$ is a foliation chart of $\Lam$, then
%there is a corresponding homeomorphism (foliation chart) $\Khat{\phi}_\alpha:
%U_\alpha\to \Khat{U}\times T_\alpha$ where $U_\alpha$ is an open set
%$\Khat{\Lam}$ and $U$ is an open set of $G$ such that the following diagram
%commutes:
%
%$$\xymatrix{
%  \Khat{U_\alpha} \ar[d]_{\pi} \ar[r]^{\Khat{\phi}_\alpha}
%                & \Khat{U}\times T_\alpha \ar[d]^{id\times\pi_0}  \\
%  U_\alpha  \ar[r]_{\phi_\alpha}
%                &    U\times T_\alpha  }$$
%where $\pi$ and $\pi_0$ are the natural projections. Similarly we define
%$\check{\Lam}$ to be the corresponding lamination with leaves
%$\L=\Ga_\L\bs G/M$.

\begin{examples}
\begin{enumerate}
\item By a theorem due to Candel (see \cite{Candel:93}), any continuous lamination by
Riemannian surfaces for which the laminated Euler characteristic is negative
admits a globally continuous conformal change of leafwise metric to ones of
constant negative curvature. The resulting lamination will be a symmetric
space lamination locally by $\RH^2=\PSL(2,\R)/SO(2)$.  Note that any
closed leaf of genus $g>1$ admits a $6g-6$ dimensional family, the
Teichm\"{u}ller space, of conformally inequivalent metrics. Each such metric
can be extended to a continuous family of metrics on the foliation by a
partition of unity argument. In this case, Candel's theorem produces a
corresponding family of pairwise metrically distinct hyperbolic laminations.

%The resulting lamination is a these metrics on the
%leaves, as well as all their derivatives, have continuous variation in the
%transverse direction.

\item If we have a group action by a semisimple group $G$ on $X$ for which
the stabilizer of each point is conjugate to $K$, then the orbits form a
symmetric space lamination. The metrics on the leaves, each one
homeomorphic to $G/K$, are induced from the Killing form on $G$. The
transversal regularity of the foliation coincides with the (transversal)
regularity of the action.

\item Another source of laminations comes from suspensions of a discrete group action.
    Let $\Gamma$ be a discrete group acting freely and properly discontinuously on a
    manifold $M$ so that $\Gamma\bs M$ is a compact (Hausdorff) manifold. Suppose
    $\Gamma$ also admits an action on a compact space $Y$ by homeomorphisms, which we
    indicate as a right action. Then we can form the compact space $X=Y\times_\Gamma
    M=(Y\times M)/\sim$ where $(y,z)\sim (y\gamma, \gamma z)$ for all $\gamma\in \Gamma$.
    The space $X$ carries the structure of a compact lamination with a global transversal
    $Y$ whose leaves are the images of $M$ in the quotient space $X$. If $M=G/K$ then the
    leaves are locally symmetric spaces with each $\Ga_L<\Ga$.
\end{enumerate}
\end{examples}

Throughout the remainder of this paper, $\Lam $ will always denote a continuous (that is
$C^{\infty,\infty,0}$) locally symmetric lamination on a compact space $X$, unless otherwise
specified.

\subsection{Harmonic measures}\label{subsec:harmonic}

Each leaf $L$ of $\Lam$, being a Riemannian manifold, has a  Laplace-Beltrami operator
$\Delta_L$. If $f:X\rightarrow\R$ is a function of class $C^2$ in the leaf direction and
$x\in\Lam$, we define $ \Delta_\Lam f(x)=\Delta_L f|_L(x),$ where $L$ is the leaf passing
through $x$ and $f|_L$ is the restriction of $f$ to $L$.  For some purposes it is convenient
to extend $\Delta_\Lam$ to the functions which leafwise belong to the Sobolev space
$H^2(L,\R)$. By the standard regularity theory for elliptical operators, such an extension
would not enlarge the class of $L^2$ harmonic functions.

Let $C^2(\Lam)$ denote the continuous leafwise $C^2$ functions on $\Lam$.  We say that a
Radon measure $\mu$ is \emph{harmonic} if $\Delta_\Lam \mu=0$ weakly; i.e. if $\int
\Delta_\Lam f\, d\mu=0$ for all $f\in C^2(\Lam)$.

If we only assume that $\Lam$ is transversely measurable ($C^{\infty,\infty}$ measurable),
then we will call $\mu$ harmonic only if it satisfies the more stringent condition that $\int
\Delta_\Lam f\, d\mu=0$ for all Borel measurable leafwise $C^2$ functions $f$. The more
stringent requirement is necessary to guarantee that all compactly supported functions on a
given leaf are included as test functions.

\begin{remark}
Since a general measurable leafwise $C^2$ function $f$ will be unbounded, we
cannot guarantee the existence of local extrema for $f$. The latter is a key step
in the the application of the Hahn-Banach theorem to show existence of
harmonic measures in the transversely continuous case (see Lemma 3.4 and
Theorem 3.5 of \cite{Candel:03}). Consequently, we cannot guarantee the
existence of harmonic measures in the transversely measurable setting.

%We also wish to point out that in the case of higher rank symmetric spaces, $G/K$, there are at least two divergent traditions governing the notion of harmonic functions. We have followed the common differential geometric notion that a harmonic function is one which vanishes under the Laplacian, or equivalently have the ususal averaging property (see e.g. \cite{Helgason77}. As seen in the relevant references below, this is the most common notion for harmonic measures on foliations as well. However there is a, perhaps more algebraic, tradition in which harmonic functions are required to vanish under all $K$-invariant second order differential operators on $G$ (see e.g. \cite{Shitak-Weitz}). In higher rank, the Laplacian is no longer the unique such operator. Nevertheless, for harmonic functions on bounded functions, it follows from the seminal work of Furstenberg \cite{Furstenberg:63} that these two notions coincide.
\end{remark}

Let $\dvol$ signify the Riemannian volume measure on each leaf $\L$. This volume is the
projection of the Haar measure of $G$ to $\Ga_\L \bs G/K$. It is known (e.g. Proposition 5.2
of \cite{Candel:03}) that every harmonic measure $\mu$ on $\Lam$ can be decomposed locally on
a flow box as $d\mu(x,t)=h(x,t)\dvol(x)\times d\sigma(t)$ where $\sigma$ is a measure on the
transversal $T$, and for $\sigma$-almost every $t\in T$, $h(\cdot,t)$ is an harmonic function
on the corresponding plaque. Moreover, this decomposition is not unique, but if
$d\mu(x,t)=h'(x,t)\dvol(x)\times d\sigma'(t)$ is another decomposition, then
$h(x,t)=\(\frac{d\sigma'}{d\sigma}(t)\) h'(x,t)$. In other words, $h$ is well defined up to a
positive constant multiple on each plaque.

As mentioned in the introduction, for the case of higher rank symmetric spaces, $G/K$, there appears a potentially distinct tradition governing what it means for a function, and by extension a measure, to be harmonic. We have up to this point followed the common differential geometric notion that a harmonic function is one which vanishes under the standard Laplacian, or equivalently possesses the usual spherical averaging property. However the most common tradition in the special setting of higher rank locally symmetric manifolds requires vanishing under the larger family of all $G$-invariant second elliptic order operators without constant term. This corresponds to having the self-averaging property on all $K$-orbits which are proper subsets of spheres. In higher rank, the Laplacian is no longer the unique such operator. Nevertheless, for bounded functions, it follows from the seminal work of Furstenberg \cite{Furstenberg:63} that these two notions coincide. We now show that essentially the same proof yields the coincidence of these two points of view in our setting as well.

\begin{proposition}
The functions $h(\cdot,t)$ and measure $\mu$ defined above are harmonic with respect to all $G$-invariant second order elliptic operators which annihilate constants. 	
\end{proposition}

\begin{proof}
The statement reduces to showing that, for a given $h(\cdot,t)$, the global well-defined extension $\hat{h}$ along plaques on the holonomy cover of the leaf is harmonic with respect to the other invariant elliptic operators.

This statement is proven in Theorem 4.4 of \cite{Furstenberg:63} for the case when $\hat{h}$ is bounded. However, the boundedness condition for this direction is used in the proof in only one place namely, to show that
\begin{align}\label{eq:limit_comm}
\frac{d}{dt}_{|_{t=0}}\int_G \tilde{h}(gg')d\mu_t(g')=\tilde{h}(g)
\end{align}
where $\tilde{h}$ is the lift to $G$ of $\hat{h}$ and $\mu_t$ is the heat semigroup for the standard Laplacian on $G$. Since this holds for compactly supported functions, we need to show that the tail of the integral vanish, as a function of $t$, at least as fast as $o(t)$.

On the other hand, Proposition 4.4 of \cite{Costa-Rechtman} shows that the $1$-form $\eta=d_\Fc \log h$ is bounded via a Harnack-type inequality. This implies that such a function $\hat{h}$ on the holonomy cover can only grow at most exponentially fast.
However, the heat kernel has decay described in equation \eqref{eq:heat_bound} that we analyze at the end. In particular its highest order decay is Gaussian in the distance and $t$, which means that the portion of the integral outside of any ball centered at the identity both converges and vanishes faster than any polynomial in $t$.  Hence we obtain \eqref{eq:limit_comm} and the proposition.
\end{proof}

\begin{remark}
The averaging property of harmonic functions under certain $K$-invariant measures $\eta$ on $G$ translates into an invariance of a harmonic measure $\mu$ under convolution by $\eta$, namely $\eta\ast \mu=\mu$. Such measures $\mu$ are called {\em $\eta$-stationary}, and were introduced and explored in \cite{Furstenberg:63b}.  While it is more common to consider stationary measures on a boundary for a random walk generated by $\eta$, in our current setting harmonic measures are stationary under the measure generating leaf-wise Brownian motion, a point of view that we will explore shortly.
\end{remark}

Later on, we will need to generalize the principle behind the above extensions
to certain operators on measures. Let $\mc{A}$ and $\mc{A}'$ be two
subspaces of $\mc{B}(X)$, the space of Borel functions on the underlying space
$X$ of $\Lam$,  and let $\mc{M}$ and $\mc{M}'$ be two subspaces of signed
Borel measures on $X$. Given an operator $D:\mc{A}\to\mc{A}'$, if for any
$\mu\in\mc{M}$ there is a unique measure $D(\mu)\in \mc{M}'$ satisfying
$$\int_{\Lam} f d\,D(\mu)=\int_{\Lam} D(f) d\mu\quad\text{for all}\quad f\in \mc{A},$$
then we can extend $D$ to an operator $D:\mc{M}\to\mc{M}'$. For instance, if
$D$ is (continuous) linear, and bounded by some norm, and $\mc{A}$ is dense
in $C(X)$, then $D$ extends to the Radon measures by an application of the
Hahn-Banach theorem followed by the Riesz theorem.

%Let $\mc{F}$ be the space of all leafwise $L^2(dvol_L)$ Borel measurable
%functions $f:\Lam\rightarrow\R$, and l
\section{Proof of Theorem \ref{thm:main}}\label{sec:proof}

%The maximal compact subgroup $K$ is the stabilizer of a some point $p\in \til \L$.
Let $p\in G/K$ denote the identity coset of $K$. Using the
Iwasawa decomposition we may express $G/K$ as $N\rtimes A$
where $A$ and $N$ are from Section \ref{sec:prelims}. The group
$N=N_\zeta$ is the unipotent group corresponding to the
nilradical which fixes some (regular) point $\zeta\in G/B$ and
similarly $A=A_{\zeta}$ also fixes $\zeta$. Under this
identification the volume form on $G/K$, which is itself just
the projection of the Haar measure on $G$ to $G/K$, coincides
with the left Haar measure on $N_\zeta \rtimes A_\zeta$.  This
in turn coincides with the left Haar measure on $N_\zeta
\rtimes A_\zeta$ for any choice of $x$ and $\zeta$ under the
identification of this simply transitive subgroup with $G/K$.
However, unlike the semisimple group $G$ and the compact group
$K$, the solvable group $N_\zeta \rtimes A_\zeta$ is not
unimodular. The modular function $\delta$ on $N_\zeta \rtimes
A_\zeta$ may be expressed in terms the Poisson kernel $\kk$.
More precisely, $\delta(g)=\kk(p,gp,\zeta)$ for $g\in N_\zeta
\rtimes A_\zeta$. Hence we can write the right Haar measure
$m^R_{N_\zeta \rtimes A_\zeta}$ as the product of
$\kk(p,gp,\zeta)$ and the left Haar measure $m^L_{N_\zeta
\rtimes A_\zeta}$ on $N_\zeta \rtimes A_\zeta$.

Let $m_{G/B}$ be the projection of Haar measure on $K$ to $K/M=G/B$. (Note
that we may make the identification $K/M=G/B$ since the action of $K=K_p$,
while fixing $p$, acts transitively on the regular points $\xi\in G/B$ and
$M=M_p^\xi$ is the stabilizer of the regular point $\xi$ in $K_p$.) The
Furstenberg boundary, $G/B$, equipped with $m_{G/B}$ is naturally
isomorphic as a measured $G$-space to the Poisson boundary of $G/K$ for the
Brownian motion, and the corresponding Poisson kernel is $\kk$ (e.g.
see\cite{Guivarch-Ji-Taylor98}).

In the flow box decomposition for $\mu$ given in the last section, each function $h(\cdot
,t)$ is a positive harmonic function on the plaque corresponding to $t$, and therefore
admits a unique extension to the plaques along any simple path in the leaf. Whenever a
leaf of the lamination has trivial holonomy, then all extensions agree to give a unique
globally defined harmonic function on $\Ga_{\L} \bs G/K$. In particular, this holds on the
holonomy cover $L'$ of any leaf $L$. Without loss of generality, in what follows we may
assume that $L$ has trivial holonomy as we will only be using $h$ locally. Passing to the
universal cover of the leaf, we again denote the $\Ga_{\L}$ invariant harmonic function on
$\tilde{L}$ by $h$. As a positive harmonic function on $G/K$ it admits a unique integral
representation on the minimal Martin boundary $\pa_{\min}G$ of $\tilde{L}\cong G/K$,
namely $h(x,t)=\int_{\pa_{\min}G}\kk(p,x,\xi) d\mu_t(\xi)$ (see I.7.9 of \cite{Borel-Ji:06}).
Here $\mu_t$ is a Borel measure associated uniquely to $h(\cdot,t)$ and $\kk(p,x,\xi)$ is the
Martin Kernel associated to the point $\xi\in B_M$ for a fixed basepoint $p\in G/K$.

In general, $\pa_{\min}G$ can be written as the disjoint union,
$$\pa_{\min}G=\coprod_B \overline{\mathfrak{a}^+_B(\infty)}$$
where $B$ runs over all minimal parabolic subgroups of $G$ and
$\overline{\mathfrak{a}^+_B(\infty)}$ is the closure of the geodesic boundary of the
corresponding positive Weyl chamber (see I.7.18 of \cite{Borel-Ji:06}). In the case that $G$
has rank greater than one, $\pa_{\min}G$ is strictly larger than the Poisson Boundary
$G/B$, as a topological space. The latter naturally sits inside $\pa_{\min}G$ as the
Furstenberg boundary $\pa_{F}G\subset \pa_{\min}G$, the union of a single point
representing the barycenter of each simplex $\overline{\mathfrak{a}^+_B(\infty)}$. If
$h(\cdot,t)$ were globally bounded then $d\mu_t=f_t dm_{G/B}$ for a unique positive
function $f_t$ on $G/B$ representing the limit values of $h(\cdot,t)$ along random walks.

It turns out that the $h(\cdot,t)$ need not be a globally bounded harmonic function. On
the other hand, the $h(\cdot,t)$ are not quite arbitrary since they are leafwise
Jacobians of a finite measure on a compact foliated space which places a restriction on
their overall growth. As explained in Section 6 of \cite{Candel:03}, the diffusion
operator generating the harmonic measure only depends on the Brownian motion on each
leaf. Since every Brownian path almost surely coverges to a point in $\pa_{F}G$
(see 6.2 of \cite{Guivarch-Ji-Taylor98}), it follows that the lifted local harmonic functions
$h(\cdot,t)$ are representable by measures on the Poisson boundary as,
$$h(x,t)=\int_{\pa_F G}\kk(p,x,\xi) d\mu_t(\xi),$$
where we now consider $\mu_t$ to be a measure with support on $\pa_F G\cong G/B$.

%We can also consider the corresponding $\Ga_L$-invariant harmonic function, again called
%$h$, on $G/K$. Since $\kk$ is the reproducing kernel on $G/B$ for harmonic functions on
%$G/K$ (\cite{Guivarch-Ji-Taylor98}\marg{Location!!}), we can express this $h$ as
%$$h(x)=\int_{G/B} \kk(p,x,\xi)f(\xi) dm_{G/B}(\xi)$$
%for some $f\in L^1(G/B)$ determined by the limits of $h$ along Brownian motions which
%converge almost surely in $G/B$.

Since $\kk(p,\ga^{-1}x,\xi)=\kk(\ga p,x,\ga \xi)=\kk(p,\ga^{-1}p, \xi)\kk(p,x,\ga \xi)$,
the $\Ga_\L$ invariance of $h$ implies that
$\frac{d\ga^{-1}_*\mu_t}{d\mu_t}(\xi)=\kk(p,\ga^{-1}p,\xi)$. In particular, if
$\Gamma_{\L}$ acts ergodically on $G/B$ with respect to $\mu_t$,  such as when the leaf
$L$ has finite volume, then uniqueness of conformal densities implies that $\mu_t$ is a
multiple of $dm_{G/B}$. This implies that $h$ is the constant function $\norm{\mu_t}$.

\begin{remark}
An alternate approach to the representation of harmonic functions arises from using hyperfunctions in the sense of Sato. (These form a class of generalized distributions.) It was shown in \cite{KKMOOT78} that every harmonic function can also be represented via a Poisson formula applied to a hyperfunction on $G/B$. In particular, the measures $\mu_t$ above also arise as hyperfunctions.
\end{remark}

Take a Borel measurable section of $K/M$ into $K$ represented by $\xi\to k^\xi$. (In
particular we have $k_\xi\zeta=\xi$.) We first note that we can express $\dvol$ on
$G/K$ as the projection under $P:G\to G/K$ of the left invariant Haar measure on the
orbit $N_\zeta \rtimes A_\zeta\cdot k^\xi$ in $G$ for any fixed $\zeta\in G/B$. This
latter measure can be written as the push forward $k^\xi_* dm^L_{N_\zeta \rtimes
A_\zeta}$. Hence fixing $\zeta\in G/B$ we may write:
\begin{align}\label{eq:harmonic}
\begin{split}
h(x,t)\dvol_{L} &=\(\int_{G/B} \kk(p,x,\xi) d\mu_t(\xi)\)\dvol_{G/K}\\
&=\int_{G/B} \kk(p,x,\xi) P_*k^\xi_*dm^L_{N_\zeta \rtimes A_\zeta} d\mu_t(\xi) \\
&=\int_{G/B}  \kk(k^\xi p,x,k^\xi\zeta) P_*k^\xi_*dm^L_{N_\zeta \rtimes A_\zeta} d\mu_t(\xi) \\
&=\int_{G/B}  P_*k^\xi_*\(\kk(p,x,\zeta)dm^L_{N_\zeta \rtimes A_\zeta}\) d\mu_t(\xi) \\
&=\int_{G/B}  P_*k^\xi_* dm^R_{N_\zeta \rtimes A_\zeta}d\mu_t(\xi).
\end{split}
\end{align}

%Since $L$ is locally $G/K$, we may express the relationship between the measures
%pointwise by,
%\begin{align*}
%h(x)\dvol_{L} &=\(\int_{G/B} \kk(p,x,\xi)f(\xi) dm_{G/B}(\xi)\) \dvol_{G/K}\\
%&=\int_{G/B} \kk(e,x,\xi)f(\xi) dm^L_{N_\xi \rtimes A_x^\xi} dm_{G/B}(\xi) \\
%&=\int_{G/B} f(\xi)  dm^R_{N_\xi \rtimes A_x^\xi}dm_{G/B}(\xi).
%\end{align*}

%We also note that if the measure $\mu$ is globally harmonic on $\Lam$ then
%the conditional measure $\tau_\L$ on each leaf $\L\in\Lam$ corresponding to
%the measurable decomposition $\Lam$, when lifted to its universal cover
%$\tilde{\L}\cong G/K$ is $\Gamma_\L$ equivariant on the left. In particular the
%$f$ above is left $\Gamma_{\L}$ invariant.

In order to lift $\mu$ we first need to choose a family of
Borel probability measures $\{\mu_{x,\xi}\}$, each
supported on its corresponding left coset $x k^\xi
M_p^\zeta=M_{xp}^{x\xi} x k^\xi$ for $x\in N_\zeta\rtimes
A_\zeta$. (These will correspond to the conditional measures under the disintegration of $\hat{\mu}$ along $\mu$.)
% left off. Fix AN to NA and always use NAK? Check Helgason and Eberlein.
We will eventually see that $\mu_{x,\xi}$ must be chosen to be the copy of the Haar measures on $x k^\xi M_p^\zeta$. We will progressively impose conditions on
these measures until they are uniquely specified. First, they must be left $\Gamma_{\L}$ equivariant, by which we mean that $\gamma_* \mu_{x,\xi}=\mu_{\gamma
x,\gamma \xi}$ for all $\gamma\in
\Gamma_{\L}$ where $L$ is the leaf of $\Lam$ containing $x$. Decomposing $\mu$ in each
foliated chart, we obtain $d\mu=d\tau_t\times d\sigma$ where $\sigma$ is a measure on the
transversal $T$ and $\tau_t$ is a measure on the plaque passing through $t\in T$. Since
$\mu$ is a harmonic measure, the measure $\tau_t$ locally has the form $h \dvol_{L}$ for
an $\Delta_L$-harmonic function $h$, as described above. Usually, $K$ as an $M$ bundle
over $K/M$ is not trivial. (Consider, for instance, the case when $L=\RH^n$ in which case
$G=PO^+(n,1)$, $K=SO(n)$, $M=SO(n-1)$ and $K/M=S^{n-1}$.) However, $K$ still admits a
product structure measurably, and so if we decompose $G$ either locally (continuously) or
globally (continuously off of a measure zero subset) as $G= N\times A\times M\times K/M$
then the above discussion shows that we can describe all of the lifts of $\mu$ which are
left Haar measures on each orbit of $N \rtimes A$ as
$$d\hat\mu(x,k,\xi,t)= d\mu_{x,\xi}(k)
d\mu_t(\xi) dm^R_{N_\zeta \rtimes A_\zeta}(x)  d\sigma(t),$$ where we have identified the support of $\mu_t$ with $K_p/M_p^{\zeta}=G/B$.

If we wish to have $\hat\mu$ invariant under the right action of $N_\zeta \rtimes
A_\zeta$, we first require that the measures $\mu_{x,\xi}$ satisfy
$g_*\mu_{x,\xi}=\mu_{xg,\xi g}=\mu_{x g,\xi}$ where the initial push-forward is under the
right action by $g\in N_\zeta \rtimes A_\zeta$. Second, we need that $\mu_t\sigma$ is
right invariant under $N_\zeta \rtimes A_\zeta$, but this follows from the fact that if
$t'$ is the transversal point corresponding to the plaque containing $tg^{-1}$, then
$\frac{d\mu_t}{dg_*\mu_{t}}=\frac{d\mu_t}{d\mu_{t'}}$ is the constant function on $G/B$
with value $\frac{dg_*\sigma}{d\sigma}(t)$, as explained previously.

%This follows from the fact that the projection of Haar measure on $K_x/M_{x}^\zeta$ is
%sent to the projection of the Haar measure on $K_{xg}/M_{xg}^\zeta$ under the right
%action by $g\in N_\zeta \rtimes A_\zeta$.

If we further require that $\hat\mu$ be invariant under the right action by
$M_{p}^\zeta$, then since the $\mu_t$ measures on $K_p/M_{p}^\zeta$ are already right
$M_{p}^\zeta$ equivariant, one only needs to require that the measures $\mu_{x,\xi}$ be
right equivariant. However, the only right invariant measure on $x k^\xi M_p^\zeta$ for
any ${x,\xi}$ is the (right=left) Haar measure, $m_{M_p^\zeta}$, so  $\mu_{x,\xi}=(x
k^\xi)_*m_{M_p^\zeta}$. Note that this is compatible with the right action of the group
$N_\zeta \rtimes A_\zeta$ which is transitive on $x\in G/K$.

%By the construction the measure $\hat\mu$ is locally invariant under the right action of
%$N \rtimes A$, at least for elements sufficiently close to the identity so that the orbit
%of a point stays within its plaque. Now we wish to extend this globally. We can lift a
%foliated chart from $\Lam$ to $\hat\Lam$ by taking the product with $K$ since on each
%plaque the corresponding $K$-bundle is trivial, while this is not an honest chart for
%$\hat\Lam$ it is more convenient. On the overlap of two such charts we require that the
%$K$ fibers agree, which we can always arrange by using an appropriate subdivision of the
%atlas of $\Lam$ such that the union of two charts and their intersections are topological
%balls. \marg{correct this}

However there is no problem to do the extension, since this
right action, which is globally defined on the entire
lamination $\hat\Lam$, agrees on the overlap of two charts
which have the same form in $\hat\Lam$ as in $\Lam$, except for
the factor $K$. Moreover note that  and therefore the action
can be extended to all of $\hat\Lam$. Hence there is a unique
right $B$-invariant lift $\hat\mu$ of $\mu$. This finishes the
proof of one direction of Theorem \ref{thm:main}.

From the above identifications, we immediately obtain the following corollary.
\begin{corollary}\label{cor:fiber}
The $N \rtimes A$ right invariant Borel measures $\hat\mu$ on $\hat\Lam$ are in one to one
correspondence with measurable sections of the bundle  $E\to G/B$ whose fiber over
$\xi \in G/B$ is the space of Borel probability measures on $M_{p,\xi}$.
\end{corollary}

\subsection{The Projection of an $N \rtimes A$ Invariant Measure}

The purpose of this section is to prove the second half of the statement in the theorem.

\begin{proposition}
  The projection of a measure on $\hat\Lam$ which is right $N \rtimes A$ invariant is
  harmonic.
\end{proposition}

\begin{proof}
We may write a right  $N_\zeta \rtimes A_\zeta$-invariant measure $\hat\mu$ as before
as
$$d\hat\mu(x,k,\xi,t)= d\mu_{x,\xi}(k)
d\mu_t(\xi) dm^R_{N_\zeta \rtimes A_\zeta}(x)
d\sigma(t),$$ with the additional requirement that the measures $\mu_{x,\xi}$
satisfy $g_*\mu_{x,\xi}=\mu_{xg,\xi}$ for all $g\in \Gamma_\L$.
% left off
Since the measures $\sigma$ are preserved and the measures $\mu_{x,\xi}$ are killed
under the projection, it remains to verify that the projection is locally $h\dvol_L$ on the
given plaque of the leaf $L\in\Lam$. However, this is precisely what was checked in
Equation \ref{eq:harmonic}, but starting from the last line and working backwards.
\end{proof}

\begin{remark}
Note that it is not necessary that the measure $\hat\mu$ be $N_\zeta \rtimes A_\zeta$
invariant in order for it to be a lift of a harmonic measure. However, it is a bit
cumbersome to precisely describe the fiber over a harmonic measure if it does not
possess the global invariance property.
\end{remark}

This also completes the proof of Theorem \ref{thm:main}.

\section{The lift of a harmonic measure}\label{sec:lift}

\subsection{Heat Kernels}
The laminated heat operator $\scr{H}=\frac{d}{dt}- \Delta_\Lam$ has
fundamental operator solution given by $D_t=e^{t\Delta_\Lam}$, which is to be
interpreted as a formal power series in $\Delta_\Lam$. Note that the collection
$\{ D_t\} _{t\geq 0}$ forms an abelian semigroup known as the
\emph{laminated heat semigroup}.

The  heat diffusion along the leaves can be described in terms of the
fundamental solution of the heat equation known as the \emph{laminated heat
kernel} and denoted by $p_t:\Lam\times \Lam\rightarrow
\R$ for $t\in \times (0,+\infty )$. It is given by
\begin{equation*}
p_t(x,y)=
\begin{cases}p_t^{\L}(x,y),& \mbox{if}\ x,y\ \mbox{belong to the same leaf}\ L\\ 0,& \mbox{otherwise,}
\end{cases}
\end{equation*}
where $p_t^{\L}$ is the heat kernel on $\L$. The laminated heat semigroup can
be expressed as
$$D_tf(x)=\int_{\L_x} p_t(x,y)f(y)\, dy,$$
for all $f\in C^2(\scr{L})$, where $\L_x$ is the leaf through $x$ and $dy$ represents
the Riemannian volume measure on $\L_x$.

A measure $m$ on $\Lam$ is \emph{ergodic} with respect to the holonomy
pseudogroup of the lamination if $\Lam$ can not be partitioned into two
measurable leaf-saturated subsets having positive $m$ measure. Lucy Garnett
(\cite{Garnett:83}) proved the following Ergodic Theorem for harmonic
measures, a nice interpretation of which can be found in \cite{Candel:03}.

\begin{theorem}\label{thm:Garnett}
Let $\mu$ be any harmonic measure on $\Lam$. For $\mu$-almost every
$x\in \Lam$, the limit of Krylov--Bogolyubov means
$$\tilde{\delta_x}:=\lim_{n\to\infty}\frac{1}{n}\sum_{t=0}^{n-1}D_t\delta_x$$
exists and is an ergodic harmonic measure. Moreover, every ergodic harmonic
measure arises this way.
\end{theorem}

\begin{remark}
  By Theorem 7.3 of \cite{Candel:03} we may instead define $\tilde{\delta_x}$ by continuous
  time averages $\tilde{\delta_x}:=\lim_{T\to\infty}\frac1T\int_0^T D_t\delta_x dt$.
\end{remark}

\subsection {The  Borel action on the canonical $K$-bundle over $\Lam$}

Geometrically, we can realize the $\hat{\Lam}$ as a sub-bundle of the full tangential frame
bundle of $\Lam$ as follows. Let $F\Lam$ be the bundle of all orthonormal frames of tangent
vectors to the leaves of $\Lam$. In particular each fiber is homeomorphic to $O(n)$.  Since
an isometry is uniquely determined by its derivative action on a frame, $G$ acts freely on
$F\tilde{\L}$ for each leaf $\L\in \Lam$. The portion of the orbit of a single frame $F\in
F_x\tilde{\L}$ in the fiber over a point $x\in \tilde{\L}$ is precisely $K_x\cdot F$, and in
particular is homeomorphic to $K_x$. However, there are many such orbits in the fiber sitting
over a single point of the leaf $\tilde{\L}$; essentially, there is one for each unit vector
in a closed Weyl chamber modulo boundary identifications. In order to create a bundle on all
of $\Lam$ we need to choose a canonical section of reference fibers  along a transversal. We
do this as follows.

By transverse continuity (resp. transverse $\nu$-measurability) of the metric we may make a
transversely continuous (resp. transversely measurable) identification of $T_t\L$ with
$\mf{p}$ at each point $t$ the transversal such that the set of regular vectors and singular
vectors vary transversely in the same fashion. (Note that the entire singular set is
determined by the metric alone and does not depend on the any choices like that of the Cartan
subspace.) In particular the space of all Weyl chambers in $\mf{p}$ under the identification
vary in the same fashion. Hence the action of $G$ will be compatible with this choice of
section in the sense that whenever $t,t'\in T$ belong to the same leaf $\L$ and $v\in T_t\L$
and $v'\in T_{t'}\L$ are any two vectors which are identified to the same element of
$\mf{p}$, then the left invariant field of $v'$ relative to the identification at $t'$ has
value at $t$ which differs from $v$ by an element of $K_t$. This follows from the fact that
the action of $G$ and in particular of $\Gamma_\L$ on left invariant fields of $\tilde{\L}$
acts at any point $x$ by elements of $K_x$. Now we choose any frame $F_o\in\mf{p}$ and take
its $K$-orbit. By the identification along the section we obtain a $K_t$ fiber at each point
$t\in T$. Taking $G$-orbits this extends analytically to each $\tilde{\L}$. As explained
above, this descends to a bundle over the quotient $\L$ under the action of $\Gamma_\L$. By
compatibility of the section, the resulting bundle is a transversely continuous (resp.
measurable) sub-bundle of $F\Lam$. Thus we can realize $\hat{\Lam}$ as a sub-bundle of
$F\Lam$.

Note that this construction does not give the stronger statement of the
existence of a section $\sigma:\Lam\to F\Lam$. As in the case of a single leaf,
existence of sections requires the vanishing of certain characteristic classes. As
a simple example,  this above construction also applies to a single semisimple
group of compact type such as $G=SO(n+1)$ in which case $K=SO(n)$ and the
$K$ orbit is $S^n=SO(n+1)/SO(n)$. With $n=2$, for instance, the absence of
sections $\sigma:S^2\to T^1S^2 \cong FS^2$ is well known.

The \emph{laminated geodesic flow} is the flow $\phi_t$ that, restricted to the
unit tangent bundle of a leaf $\L$ of $\Lam$, coincides with the geodesic flow
in $T^1\L$.  If we identify the tangent space of $[K]\in G/K$ with the polar
subspace $\mf{p}$, the orthogonal complement of $\mf{k}$ in the Lie algebra
$\mf{g}$, then we can write general elements of $T^1\L=\Ga\bs T^1 (G/K)$ as
$\Ga g \cdot v$ for some $g\in G$ and $v\in \mf{p}$. Moreover, the geodesic
flow is given by the map $\Ga g\cdot v\mapsto
\Ga g\exp(tv)\cdot v$, see \cite{Bekka-Mayer:00b}. Except in the rank one case, we cannot
identify $T^1 (G/K)$ as a homogeneous quotient of $G$ itself, so we cannot
describe this flow by any (single) right action of a one parameter subgroup
$T<G$.

We can similarly define the \emph{laminated Weyl chamber flow} on
$\check{\Lam}$ and the \emph{laminated Borel action} on $\hat{\Lam}$ by
combining the corresponding leafwise actions on $\Ga_\L\bs G/M$ and
$\Ga_\L\bs G$, respectively. These are given by global actions consisting of
right multiplication by elements of $A$ on $\check{\Lam}$ and elements of
$B<G$ on $\hat{\Lam}$, respectively. The standard unipotent action given by
the restricted right action of $N<B$ on $\Ga\bs G$ generalizes the stable
horocycle flow on $\Ga\bs PSL(2,\R)= T^1(\Ga\bs\RH^2)$, for which case
$N\cong \R$. Both the Weyl chamber flow on $\check \Lam$ and the Borel
action on $\hat\Lam$ have the same regularity as the full right $G$-action on
$\hat\Lam$.

%as a consequence of Candel's theorem (see
%\cite{Candel1}).

\subsection{A measure invariant under the action of the affine group}

Let $\pi:\hat\Lam\rightarrow\Lam$ be the canonical projection, and consider
a harmonic probability measure $\mu$ on $\Lam$. In \cite{Martinez:06}, there
is a construction in the case of $G=\SL(2,\R)$ that produces a measure
$\hat\nu$ on $\hat\Lam$ which is invariant under the stable (or the unstable)
horocycle flow and such that $\pi_*\hat\nu=\mu$. In this section we will
construct $\hat\nu$ in a manner analogous to this construction, and it will
follow from the properties that $\hat\nu=\hat\mu$.

%It will be more convenient for us to consider the measure $\nu$ which is invariant under
%the \emph{unstable} horocycle flow and which projects onto $\mu$.

We consider $G$ as the particular sub-bundle of the full frame bundle
mentioned earlier. We begin by using a fixed point $p$. As in
\cite{Martinez:06}, two simplifying assumptions shall be made, that carry no
loss of generality:

\begin{enumerate}
\item That $\mu$ is ergodic; in particular, by Theorem \ref{thm:Garnett} there is an
    $p\in\Lam$ ($\mu$-almost any point will do) such that $\mu$
    arises as a limit of Krylov--Bogolyubov sums starting
    from an initial Dirac mass at $p$; and

\item that the point $p$ belongs to a leaf which is simply connected. (If it does not, we
    consider the universal cover of the leaf through $p$, as explained in
    \cite[pg.857]{Martinez:06}.)
\end{enumerate}

\subsection{The Construction of $\hat\nu$}

For any natural number $n\geq 1$, let $\delta^{(n)}_p$ be the
measure given by the following Krylov--Bogolyubov sum,
$$\delta_p^{(n)}=\frac{1}{n}\sum _{t=0}^{n-1} D_t\delta_p.$$
More prosaically, $\delta^{(n)}_p$ is the probability measure such that, for every
continuous function $f$ in $\Lam$,
$$\int_{\Lam} f\, d\delta^{(n)}_p =\frac{1}{n}\sum_{t=0}^{n-1}
D_tf(p)=\frac{1}{n}\sum_{t=0}^{n-1}\int _{L_p} p_t(p,y)f(y)dy.$$ With this notation, under
our assumptions we have $\mu=\lim_n \delta^{(n)}_p$ for a generic choice of $p$.

Consider the leaf $\L$ of $\Lam$ passing through $p$. By our
second assumption above, $\L$ is a global symmetric space
identified with $G/K_p$. Let $\mc{S}\subset \L$ denote the
(measure zero) set of points in $\L$ whose outward pointing
radial vector from $p$ is a singular vector in $T\L$ under this
identification. By convention, we also assume $p\in \mc{S}$. In
other words, $\mc{S}$ consists of the union of points in the
$p$-orbit of the chamber walls in all Cartan subgroups. Let
$\mc{V}:\L\setminus \mc{S}\rightarrow \check\Lam$ be the
``radial'' Weyl chamber field defined as follows. For each
point $x\in \L\setminus \mc{S}$ there is a unique unit speed
geodesic passing through $x$ starting from $p$ that ends at a
regular point $\xi(x)\in\bd \L$, where here $\bd L$ represents
the geodesic (ideal) boundary of $L$. Set $\mc{V}(x)\in
G/M\cong \check \L$ to be the unique Weyl chamber that, when
viewed as a subset of $\mf{g}$, has exponential image in
$G/K\cong \L$ with vertex at $x$ and $\xi(x)$ on its ideal
boundary. We remark that $\mc{V}$ does not extend continuously
to all of $\L$, but we can take a measurable extension which is
continuous almost everywhere.

% No only invariant under \bar{A^+} otherwise crosses singular set just like regular
% spherical field only invariant under positive geodesic flow.
% (negtive will cross giving opposite"Weyl Chamber" not outward
% vector)

Now using the Cartan decomposition $G=K_p \overline{A_\zeta^+}
K_p$ we can identify points in $\check\L\cong G/M_p^\zeta$ with
$K_p \overline{A_\zeta^+}K_p/M_p^\zeta$ so that any Weyl
chamber $\mc{C}$ with base point over $x=k a K_p\in G/K_p$ can
be written as $\mc{C}=k a k' M_p^\zeta$. (Note that $k \cdot
\xi=\xi(x)$.) The defining condition for $\mc{V}(x)$ is that
$k'=e$ in other words,  $\mc{V}(x)=k a M_p^\zeta$. Since
$M_p^\zeta$ is the centralizer of $A_\zeta$, $a_o\in
\overline{A_\zeta^+}$ acts on the right by $\mc{V}(x)a_o=k a
a_o M_p^\zeta=\mc{V}(xa_o)$ where $xa_o$ means the point
$xa_o=kaa_oK_p$. Here we have used that $\overline{A_\zeta^+}$
is closed under multiplication.
%In fact, all of $A_\zeta$ acts on the right of
%$G/M_p^\zeta$ and if $a_1\in A_\zeta$ is $a_1=w a_o$ for some
%$a_o\in \overline{A_\zeta^+}$ and some choice of representative
%$w\in M'$ of the unique Weyl transformation, then
%$\mc{V}(x)a_1= k a wa_o M_p^\zeta= k w a a_o  M_p^\zeta=\mc{V}(kw a
%a_o K_p)$. In other words $A_\zeta$ acts the same as
%$\overline{A_\zeta^+}$ except for a permutation of the chambers.
Consequently,  the action of $wa_o\in \overline{A_\zeta^+}$ sends the the section
$\mc{V}$ into, but not onto, itself. For since $a_o,a\in \overline{A_\zeta^+}$, the
distance of the basepoint of the Weyl chamber to $p$ is always increased by $d(p,a a_o
p)-d(p,a p)\geq 0$, where $d$ is the distance function on $L$.

This expansive action generalizes the positive geodesic flow
acting on the unit radial vector field emanating from $p$ in
the rank one case. Note that it is well defined and continuous
only because $\mc{V}$ is defined on $\L\setminus \mc{S}$.
% Note that if we extend $\mc{V}$ to be lower semicontinuous on
%all of $\L$, the equivariance of $\mc{V}$ under the, now
%discontinuous along $\mc{S}$, right action of
%$\overline{A_\zeta^+}$ is still preserved.
For convenience we shall extend $\mc{V}$ to the measure zero
set $\mc{S}$ to obtain a measurable section
$\mc{V}:\L\to\check\L$. Now that we have the frame field
$\mc{V}$, we are ready to construct $\hat\nu$.

%If we denote the algebraic barycenter of $A_\zeta^+$ by $\rho=\sum_{\Lambda^+}$, then
%we can choose an arbitrary reference frame The initial vector will be the unit radial
%vector field pointing outwards from $p$. Choose a fixed reference frame $F_{x_o}$ at the
%algebraic barycenter of the root vectors $x_o\in A^+$ whose initial vector is radial. We
%define the frame $F_x$ at any other point $x\in A^+$ with $d(p,x)=d(p,x_o)$ by taking
%the initial vector to be the radial one, and translating $F_{x_o}$ to $y$, projecting it to
%the tangent space of the entire sphere centered at $p$ through $x$ and then finally
%renormalizing via the Gram-Schmidt process to obtain a frame. This is well defined since
%$A^+\cap S(p,d(p,x_o))$ lies strictly in a hemisphere centered at $x_o$. Now we extend
%this frame to the rest of $A^+$ by parallel translate along the radial geodesics (linear
%rays). Next we extend this to all of $L_p$ by taking the $K$-orbit.

%. ; i.e. $\mc{V}(y)=(\gamma(s),\dot\gamma(s))$ if $\gamma$ is the unit speed geodesic such
%that $\gamma(0)=x$ and $\gamma(s)=y$. (Here $s$ is the distance from $x$ to $y$
%measured on $L_p$.)

Define $\mu_n=\mc{V}_*\delta^{(n)}_p$, where recall we are treating
$\mc{V}$ as a mapping from $\L$ to $\check \L$. This gives a sequence of
probability measures on the compact space $\check\Lam$. Let $\check\nu$ be
any limit point of the sequence $(\mu_n)$ in the sense of the weak-* topology.
(It is not difficult to see that in fact $\check\nu$ is the unique limit of the
$\mu_n$, but this is unessential for our argument and will follow anyhow from
the properties after the fact.) Finally we define $\hat\nu$ to be locally the
product $d\hat\nu=d\check\nu\times dm_{M}\times d\sigma$ where
$\sigma$, as before, is the transversal measure of the original harmonic
measure $\mu$ and $m_M$ represents the Haar measure on each of the fibers
of the fibration $\hat\Lam\to
\check\Lam$ which are canonically identified with $M$.

%Then, as was proved in \cite{Martinez}, $\nu$ is invariant under the unstable horocycle
%flow. And $\pi_*\nu=\mu$ since $\pi_*$ is a continuous map from the space of finite
%measures on $\hat\Lam$ to that of finite measures on $\Lam$.

\begin{theorem}\label{thm:same}
We have $\hat\nu=\hat\mu$.
\end{theorem}

%In particular, this theorem says that $\nu$ is invariant under the geodesic flows.
%Therefore, any harmonic measure is the projection of a measure invariant under both
%flows. This statement and its converse which is quoted in the Introduction constitute our
%Main Theorem.

\begin{proof}
First note that $\pi_*\hat\nu=\mu$ since $\mc{V}$ is a section, and under our
assumptions $\mu=\lim_n \mu_n$.

It remains to show that $\hat\nu$ is invariant under the full right action by the minimal
parabolic subgroup $B=N_\zeta A_\zeta M_p^\zeta$. The measure $\hat\nu$ is right
$M_p^\zeta$ invariant since the Haar measures are, and the right action of this group is
trivial on
$G/M$ and thus on $\check\nu$. % this is the maximal subgroup of $K_p$ that fixes the
%$A^+$ that appeared in the definition of $\mc{V}$. From the construction of $\mc{V}$, it
%is held invariant under this right action.

Next we show that $\hat\nu$ is invariant under the right action by $N_\zeta$. We recall
that $\mc{V}$ is invariant under the left action by $K_p$. Observe that at a point $g\in
G$, the $K_p$ orbit $K_pg$ has tangent space at $g$ which is $(R_g)_*\mathfrak{K}_p$
and the pullback of this to the identity is $(L_g)^*(R_{g})_*\mathfrak{K}_p$ which is the
tangent space at the identity to the conjugate stabilizer group $K_{g^{-1}p}$. This group
acting on the left on $p$ has orbit which is contained in the  sphere of radius $d(p,gp)$
in $G/K$. In fact this orbit is the entire sphere in rank one and has codimension one less
than the rank in general. Setting $A^+_\zeta=\exp(\mf{a}^+_\zeta)$, take $g=a_\zeta\in
A^+_\zeta$. As $a_\zeta$ tends to infinity, which is to say that $a_\zeta\to \xi\in \bd \L$,
the spheres centered at $gp$ of radius $d(p,gp)$ limit to the horosphere based at $p$
tangent to $\xi$. Both these orbits and the left orbit of $N_\zeta$ admit the same
transversals in the flat, but to show that the orbits converge one notes that the
homogeneous spaces $K_{g^{-1}p}/M_{g^{-1}p}^\zeta$ converge to $N_\zeta$ since the
quotient spaces of the subalgebras converge to the subalgebra $\mf{n}_\zeta$ in
$\mathfrak{g}$. In other words $a_\zeta^{-1} K_p a_\zeta$ tends to $M_p^\zeta N_\zeta$.
In particular, $a_\zeta^{-1} K_p a_\zeta M_p^\zeta=a_\zeta^{-1} K_p M_p^\zeta
a_\zeta=a_\zeta^{-1} K_p a_\zeta$ tends to $M_p^\zeta N_\zeta M_p^\zeta=M_p^\zeta
N_\zeta$. However, it is not the case that for any particular $k\in K_p$ there is an $n\in
N_\zeta$ such that $a_\zeta^{-1} k a_\zeta M_p^\zeta$ tends to $M_p^\zeta n$ since the
right $N_\zeta$ action does not preserve $M_p^\zeta$ cosets. (One can see this easily by
noting the difference in the dimensions of the left orbit of $p$ by these cosets in $G/K$.)

Since the heat kernel is $G$-equivariant, $p_t(gp,hp)=p_t(p,g^{-1}hp)$, and using the
Cartan decomposition $G=K_p\overline{A^+}K_p$ we may write $p_t(p,kak'\dot
p)=p_t(p,ka\cdot p)$. Moreover on a symmetric space $p_t(p,\cdot )$ is symmetric under
the action of $K_p$, so that $p_t(p,ka\cdot p)=p_t(p,a\cdot p)$ (see \cite{Anker-Ji:99}). In
other words, the heat kernel just depends on $a\in A^+_\zeta$. Hence the conditional
probability measures $\delta^{(n)}_p$ on any left $K_p$ orbit in $G/K$ are just the
projections of the unit Haar measures on $K_p$. The lift under $\mc{V}$ to $G/M$ does
not change this. In particular, the conditional measures for $\mu_n$ on $K_p\cdot
 a\cdot\mc{V}(p)$ do not depend on $n$ and so we denote them by $\alpha_a$. We
observe that these also coincide with the conditional probability measures of
$\mc{V}_*D_t\delta_p$ for all $t\geq 0$. In fact the $\alpha_a$ do not really
depend on $a$ either, except to indicate the particular orbit they live on.

Since we lift these to $G$ by the bi-invariant measure on $M_p^\zeta$, the radial
conditional measures of $\hat\nu$ are the Haar measures on $K_p$. Thus their limit
under conjugation by $a_\zeta$ is also the bi-invariant measure on the space
$M_p^\zeta N_\zeta$. This measure restricts to the (bi-invariant) Haar measures under
both decompositions $M_p^\zeta N_\zeta$ and $N_\zeta M_p^\zeta $ of this space. In
particular, these conditional measures are right $N_\zeta$ invariant, and thus so is
$\hat\nu$.

% This is equivalent to
%Now looking at the image under $\mc{V}$ of these orbits, one further lifts them to $G$
%by taking the saturation by right $M$ orbits. In $G$ the left orbit of $N_\zeta$ of an $M$
%coset of a lift in the image of $\mc{V}$ is identical to the right $N_\zeta$ orbit since
%$\mc{V}(p)$ at these points is nearly the "outward pointing" Weyl chambers normal to
%the horosphere containing the orbit $N_\zeta$. This agreement improves on ever larger
%sets the closer $g^{-1}p$ becomes to $\zeta$.

%In other words, the right orbit $g\cdot N_\zeta$ limits tangentially to the left orbits
%$K_p\cdot g$ as $gp\to \xi$. The left invariance of $\mc{V}$ by $K_p$ implies that
%asymptotically $\mc{V}$ is right $N_\zeta$ invariant. Note that on any compact subset,
%$\mu_n$ vanishes as $n\to \infty$, and the push-forward $\nu$ only depends on the
%behavior of $\mc{V}$ asymptotically in $G$.

%Now the groups
%$K_{gp}$, along with $N_\zeta$, are unimodular, and so these measures limit as above to
%the right Haar measure on the right $N_\zeta$ orbit. Combining this with the invariance of
%$\mc{V}$, it follows that $\nu$ is invariant under $N_\zeta$.

To finish the proof of Theorem \ref{thm:same} we need to show invariance of $\hat\nu$
under the right action of $A_\zeta^+$. Since the right action by $A_\zeta^+$ takes
$M_\zeta$ cosets to cosets, and preserves their Haar measures, it is enough to check the
condition on $\check\nu$.

%The conditional (probability) measures, $\alpha_{a}$, of $\mu_n$ under the
%decomposition of $\Gamma_{\L}\bs G/M$ into $K_p$ orbits of points $a\cdot
%\mc{V}(p)$ for $a\in A_\zeta^+$,  are the images of the normalized
%Haar measure on the orbits of $K_p$ which lie in the sphere $S(p,r)$ of radius
%$r=d(p,a)$.

We will consider an arbitrary element $a_o=\exp(H_o)\in A_\zeta^+$. We need to show
that $(\phi_{a_o})_*\nu=\nu$ where $\phi_{a_o}$ is simply the right action by $a_o$,
namely $R_{a_o}$ on $G/M$. (The purpose for the additional nomination $\phi_{a_o}$ is
merely to suggest the geodesic flow which it generalizes.) For this it is sufficient to
establish the following invariance under the heat flow.

For every continuous real-valued function $f$ on $\check\Lam$ and every $a_o\in
A_\zeta^+$,
\begin{align}\label{eq:diffuse}\tag{*}
\lim_{t\rightarrow +\infty}\left|\int_{\check\Lam}f\,
d(\mc{V}_*D_t\delta_p)-\int_{\check\Lam}f\circ \phi_{a_o}\,
d(\mc{V}_*D_t\delta_p)\right|=0.
\end{align}

Let $S_a=K_p\cdot a p$ denote the $K_p$ orbit of $a\cdot p$. Its normalized Lebesgue
measure can be pushed forward by $\mc{V}$ to obtain the measure $\alpha_a$,
supported on the manifold $\mc{V}(S_a)\subset \check\Lam$. However, since the
section $\mc{V}$ was constructed to be $A_\zeta$ equivariant, we may reduce the
verification of \eqref{eq:diffuse} to the corresponding condition on each conditional
measure of $\mu_n$. To determine this condition, we first note that we can write the
integral of $f$ with respect to $\mc{V}_*D_t\delta_p$ as
$$\int_{A_\zeta^+} u_t(a)\left(\int_{\mc{V}(S_a)} f\, d\alpha_{a} \right) \,
dm_{A_\zeta}(a)=\int_{A_\zeta^+} u_t(a)\left(\int_{K_p} f(k a \mc{V}(p))\, dm_{K_p}(k)
\right) \, dm_{A_\zeta}(a),$$ where $u_t(a)=p_t(p,a)\times \vol(S_a)$ for any point $a\in
A_\zeta^+$ and $dm_{K_p}$ is the probability Haar measure on $K_p$. (Recall that the
heat kernel $p_t(p,\cdot)$ only depends on $a\in A_\zeta^+$.) For $a_o\in A_\zeta^+$,
the flow $\phi_{a_o}$ takes $\mc{V}(S_a)$ to $\mc{V}(S_{aa_o})$, and
$(\phi_{a_o})_*\alpha_{a}=\alpha_{aa_o}$. Recall that the function $f$ on $\Lam$ is
necessarily bounded. Writing $g(a)=\int_{K_p} f(k \cdot a\cdot \mc{V}(p))\,
dm_{K_p}(k)$, and making use of the change of variables $a\mapsto aa_o^{-1}$, the
condition \eqref{eq:diffuse} follows from:

For every continuous bounded function $g:A_\zeta^+\rightarrow\R$ and every $a_o\in
A_\zeta^+$,
\begin{align}\label{eq:diffuse2}\tag{**}
\lim_{t\rightarrow\infty}\left|\int_{A_\zeta^+} u_t(a)g(a)\,dm_{A_\zeta}(a) - \int
_{a_oA_\zeta^+} u_t(aa_o^{-1})g(a)\, dm_{A_\zeta}(a)\right|=0.
\end{align}

Note that $A_\zeta^+$ may be decomposed, modulo measure zero sets,  into the disjoint
union $\overline{A_\zeta^+}=\(\cup_\alpha S_\alpha\)\cup a_o A_\zeta^+$ where
$\alpha$ runs over the faces of a closed Weyl chamber and the $S_\alpha$ are slabs of
the form,
$$S_\alpha=\exp\(\set{s H_o+\mf{F}_\alpha\,:\,s\in[0,1]}\)$$
where $a_o=\exp H_o$ and $\mf{F}_\alpha\subset \pa \overline{\mf{a}_\zeta^+}$ is a
face of the Weyl chamber.

Applying this decomposition and translating the domain $a_oA_\zeta^+$, the equality
\eqref{eq:diffuse2}  becomes:

\begin{align}
\lim_{t\rightarrow\infty}\left|\int_{A_\zeta^+} (u_t(aa_o)-u_t(a))g(aa_o)\,dm_{A_\zeta}(a)+
\int _{\cup_\alpha S_\alpha} u_t(a)g(a)\, dm_{A_\zeta}(a)\right|=0.
\end{align}

Finally this claim and the theorem, follow from the following proposition.
\end{proof}

\begin{proposition}\label{prop:decay}
For any $a_o\in A_\zeta^+$,
\begin{equation}
\label{eq:closeness_of_densities}
\lim_{t\rightarrow\infty}\int_{A_\zeta^+}|u_t(a)-u_t(aa_o)|dm_{A_\zeta}(a)=0.
\end{equation}
Moreover, the integral of $u_t$ on any slab $S_\alpha$ vanishes in the limit as
$t\to
\infty$.
\end{proposition}

\begin{proof}
We will express elements $a\in A_\zeta^+$ additively by $a=\exp(H)$ and
$a_o=\exp(H_0)$ for $H,H_o\in\mf{a}_\zeta^+$. The heat kernel, $p_t(p,gp)$ for $g=nak$
only depends on $a=\exp H$ so we write $p_t(p,gp)=p_t(a)$. Writing $f\asymp g$
whenever there is a universal constant $C\geq 1$ such that $\frac1C g\leq f\leq C g$,
the Main Theorem (3.1) of  \cite{Anker-Ostellari:03}, states that
\begin{align}\label{eq:heat_bound}
p_t(\exp(H))\asymp t^{-\frac{\ell}{2}}\scriptstyle{\left(\prod_{\alpha\in
\Lambda^{++}}\frac{1+\inner{\alpha,H}}{t}\(1+
\frac{1+\inner{\alpha,H}}{t}\)^{\frac{m_\alpha+m_{2\alpha}}{2}-1} \right)}
e^{-\abs{\rho}^2t-\inner{\rho,H}-\frac{\abs{H}^2}{4t}},
\end{align}
where $\ell=\dim A$ is the rank of $G$, $\Lambda^{++}\subset \Lambda^+$ is the set of
indecomposable positive roots, $\rho$ is the algebraic centroid of $\mf{a}_\zeta^+$
defined by $\rho=\frac12 \sum_{\alpha\in\La^+}m_\alpha H_\alpha$ where $H_\alpha$
are the root vectors with multiplicity $m_\alpha=\dim \mf{g}_\alpha$.

Furthermore,  one can obtain a sharp asymptotic estimate for  $p_t(\exp(H))$ when $H$
stays away from chamber walls. More precisely, Proposition 3.2 of \cite{Anker-Ji:01}
states that for any sequence $H_j\in \mf{a}^+$ such that $\inner{\alpha,H_j}\to \infty$ for
all $\alpha\in \La^{+}$ and any sequence $t_j\to\infty$, we have
\begin{align}\label{eq:heat_asymp}
\lim_{j\to\infty} \frac{1}{p_{t_j}(\exp H_j)}C_o t_j^{-\frac{\ell}{2}}\mbf{c}\(-\frac{i
H_j}{2t_j}\)^{-1}e^{-\abs{\rho}^2t_j-\inner{\rho,H_j}-\frac{\abs{H_j}^2}{4t_j}}=1,
\end{align}
where $C_o>0$ is a universal constant, and $\mbf{c}$ is the function of Harish-Chandra
which in this range satisfies the bounds,
\begin{align}\label{eq:c-poly}
\mbf{c}\(-\frac{i H}{2t}\)^{-1}\asymp \prod_{\alpha\in
\Lambda^{++}}\frac{1+\inner{\alpha,H}}{t}\(1+
\frac{\inner{\alpha,H}}{t}\)^{\frac{m_\alpha+m_{2\alpha}}{2}-1}.
\end{align}

Moreover, the volume of the $K_p$ orbit through $\exp(H)$ is
$$C_1\prod_{\alpha\in\Lambda^+}\sinh\(\inner{\alpha,H}\)^{m_\alpha}$$ for some
constant $C_1$ that depends on the normalization of the Haar measure used. As
$\abs{H}$ increases this rapidly approaches $C_1e^{2\inner{\rho,H}}$. So after
multiplying $p_t$ by the volume in the restricted regime, we obtain the asymptotic,
\begin{align*}
u_t(\exp H)\sim& C t^{-\frac{\ell}{2}}\mbf{c}\(-\frac{i
H}{2t}\)^{-1}e^{-\abs{\rho}^2t+\inner{\rho,H}-\frac{\abs{H}^2}{4t}}=\\
&C t^{-\frac{\ell}{2}}\mbf{c}\(-\frac{i H}{2t}\)^{-1}e^{-\frac{\abs{H-2\rho t }^2}{4t}},
\end{align*}
where $C>0$ is the combined universal constant. Similarly from \eqref{eq:heat_bound},
in all cases we have the bounds,
\begin{align*}
u_t(\exp H)\asymp  t^{-\frac{\ell}{2}}\scriptstyle{\left(\prod_{\alpha\in
\Lambda^{++}}\frac{1+\inner{\alpha,H}}{t}\(1+
\frac{1+\inner{\alpha,H}}{t}\)^{\frac{m_\alpha+m_{2\alpha}}{2}-1} \right)}e^{-\frac{\abs{H-2\rho t }^2}{4t}}.
\end{align*}
This estimate shows that for $\abs{H}$ large and fixed, $u_t(\exp H)$ is maximized when
$H$ is nearly in the direction of $\rho$. Moreover, $u$ roughly decreases at least until
$H$ comes within a factor of $\log \abs{H}$ of the Weyl chamber walls.

Now denote by $\mf{w}(R)\subset \mf{a}^+_\zeta$ the complement in $\mf{a}^+$ of the
set $\mf{a}^++R\rho$. Hence, $W(R)=\exp(\mf{w}(R))$ is the complement of the set
$\set{\exp(H+R\rho)\,:\, H\in\mf{a}^+}$ and we set
$A(R)=A^+_\zeta-W(R)=\exp(\mf{a}^++R\rho)$. From the definition of $\rho$, there are
constants $c_\alpha$ such that $H$ belongs to $\mf{w}(R)$ if and only if
$\inner{\alpha,H}<c_\alpha R$ for some $\alpha\in \La^+$.

Now we recall two important properties of the heat kernel, the
heat kernel uniformly tends to zero on compacta as $t\to
\infty$.)): $\int_{A_\zeta^+} u_t(a)dm_{A_\zeta}(a)=1$
independent of $t$, and $u_t$  tends to $0$ uniformly on
compacta (Theorem VIII.8 of \cite{Chavel84} and using that G/K
is a Hadamard space). Since the ratio $\frac{\vol W(R) \cap
B(x,r)}{\vol A^+_\zeta \cap B(x,r)}$ tends to $0$ as $r\to
\infty$ for each fixed $R>0$, the small scale average
comparison above implies that,
$$\lim_{t\to \infty}\int_{W(R)} u_t(a)dm_{A_\zeta}(a)=0.$$

Now we can compute
\begin{align*}
 \int_{A(R)}&|u_t(a)-u_t(aa_o)|dm_{A_\zeta}(a)=\\
 &\int_{A(R)}u_t(a)\abs{1-\frac{u_t(aa_o)}{u_t(a)} }dm_{A_\zeta}(a)=\\
 &\int_{A_\zeta^+}u_t(\exp(H+R\rho))\abs{1-C(R)\frac{\mbf{c}\(-i\frac{
H+H_o+R\rho}{2t}\)^{-1}e^{-\frac{\abs{H+H_o-(2t-R)\rho  }^2}{4t}}}{\mbf{c}\(-\frac{i
H+R\rho}{2t}\)^{-1}e^{-\frac{\abs{H-(2t-R)\rho}^2}{4t}}} }dm_{A_\zeta}(a=\exp(H)).
\end{align*}
Here $C(R)$ tends to $1$ as $R\to\infty$.

% Since we are taking the limit as $t\to \infty$ and the heat kernel vanishes on compacta (\cite{Chavel84}\marg{where exactly?}),
%the limit of the integral is the same if we remove any fixed
%compact set from $A_\zeta^+$.
Now as shown in \eqref{eq:c-poly} above $\mbf{c}^{-1}$ has uniform polynomial growth
in each $\inner{\alpha,H}$ and hence for $\abs{H}>>\abs{H_o}$ and $R>>0$ we have
that
$$C(R)\frac{\mbf{c}\(-i\frac{ H+H_o+R\rho}{2t}\)^{-1}}{\mbf{c}\(-\frac{i
H+R\rho}{2t}\)^{-1}}$$ is very close to $1$.

Putting this together and expanding the quadratic terms and making obvious
cancelations yields,
\begin{align*}
 \lim_{t\to\infty}\int_{A_\zeta^+}&|u_t(a)-u_t(aa_o)|dm_{A_\zeta}(a)=\\
 & \lim_{t\to\infty}\int_{A_\zeta^+}u_t(\exp(H))\abs{1-e^{-\frac{\inner{H,H_o}}{2t}
 +\inner{\rho,H_o}-\frac{\abs{H_o}^2}{4t} } }dm_{A_\zeta}(\exp(H))=\\
& \lim_{t\to\infty}\int_{A_\zeta^+}u_t(\exp(H))\abs{1-e^{-\frac{\inner{H,H_o}}{2t}
+\inner{\rho,H_o}} }dm_{A_\zeta}(\exp(H)).
\end{align*}

Finally the Gaussian form of $u_t$ implies that for $t$ sufficiently large that there is a
$\delta(t)$, tending to $0$ as $t\to\infty$, and an $\eps(t)$ also tending to $0$ as
$t\to\infty$ such that $1-\delta(t)$ of the mass of $u_t(\exp(H))$ is contained in the
$\sqrt{t}^{1+\eps(t)}$ ball centered at $2\rho t$, i.e.
$$\abs{H-2\rho t}<\sqrt{t}^{1+\eps(t)}.$$
However, writing $H=H'+2\rho t$ we have
$$1-e^{\inner{\rho,H_o}-\frac{\inner{H,H_o}}{2t}}=1-e^{-\frac{\inner{H',H_o}}{2t}}.$$
Furthermore, this function is on the order of at most
$\frac{1}{\sqrt{t}^{1-\eps(t)}}$ on this ball. Hence the limit of the integral
vanishes since $u_t$ is positive with unit mass.

For the last statement, note that the slab $S_\alpha$ is the exponential image of a set
contained within the Euclidean $\delta$-neighborhood of $\mf{F}_\alpha$ with width
$\delta$ at most $\abs{H_o}$. On the other hand, $H_o+ \mf{a}_\zeta^+\subset
\mf{a}_\zeta^+$ is at finite Hausdorff distance from $\mf{a}_\zeta^+$. Hence for any
$\eps>0$ and any sequence $t_i\to\infty$ the exponential image of the ball of radius
$\sqrt{t_i}^{1+\eps}$ centered at $2\rho t_i$, eventually does not even intersect
$S_\alpha$.  Since the integral of $ u_{t_i}$ on all of $A_\zeta^+$ is one, and $u_{t_i}$ is
almost entirely supported on this ball, the value of $u_t$ on $S_\alpha$ vanishes as
$t\to\infty$.
\end{proof}

%Section 2 is devoted to the proof of this theorem. $\square$

%\begin{remark} Theorem 1 says that any harmonic measure on $\Lam$ is the
%projection of a measure invariant under the action of the lower triangular group in
%$\hat\Lam$. Of course we could do the same thing for the upper triangular group,
%considering the unit inward radial vector field on $L_p\backslash \{x\}$, and then taking
%$s$ in $[-1,0]$.\end{remark}

%%%%%%%%%%%%%%%%%%%%%%%%%%%%%%%%%%%%%%%%%%
%%%%%%%%%%%%%%%%%%%%%%%%%%%%%%%%%%%%%%%%%%
%%%%%%

\section{Applications and Related Results}\label{sec:examps}

%%%%%%%%%%%%%%%%%%%%%%%%%%%%%%%%%%%%%%%%%%
%%%%%%%%%%%%%%%%%%%%%%%%%%%%%%%%%%%%%%%%

As a simple application of Theorem \ref{thm:main}, we prove unique ergodicity
for the action of $B<G$ on the frame sub-bundle $\hat\Lam$ of certain
foliations by locally symmetric spaces $\Gamma_\L\bs G/K$.  We will need the
following recent results on unique ergodicity for harmonic measures.

\subsection{Transversely conformal foliation by symmetric spaces}

A (transversely) conformal foliation of codimension $q$ is essentially a foliation
whose local holonomy maps are restrictions of conformal maps of the sphere
$S^q$, where local transversals have been identified with subsets of the sphere.
(For precise definitions see \cite{Vaisman:79}.) This is a large class of foliations
including all codimension one foliations and complex codimension one
holomorphic foliations. In \cite{Deroin-Kleptsyn:07}, Deroin and Kleptsyn
prove the following:

\begin{theorem*}
Let $(M,\Fc)$ be a compact manifold together with a transversely conformal
foliation and Riemannian metrics on the leaves that vary continuously in $M$.
Let $\mathcal{M}$ be a minimal set for $\Fc$ (that is, a closed saturated subset
of $M$ where all leaves are dense.) If $\mathcal{M}$ has no transverse
holonomy-invariant measure, then it has a unique harmonic measure.
\end{theorem*}

Combining this with our main result immediately yields:

\begin{corollary}
Let $(M,\Fc)$ be a compact manifold together with a transversely conformal
foliation by locally symmetric spaces $\Ga_\L\bs G/K$. Let $\mathcal{M}$ be a
minimal set for $\Fc$. If $\mathcal{M}$ has no holonomy-invariant measures,
then the action of $B$ on $\hat{\mc{M}}$ is uniquely ergodic.
\end{corollary}

This generalizes the corresponding result for surfaces in
\cite{Bakhtin-Martinez:08}.

\subsection{Holonomy Invariant Measures}\label{sec:invariant_meas}

We now give the proof of Theorem \ref{thm:G-inv}, which we state again for
convenience.

\begin{theorem}[\cite{Muniz-Manasliski:09}]
There is a natural bijective correspondence between $G$-invariant measures on $\hat\Lam$ and
invariant transverse measures on $\Lam$.
\end{theorem}

\begin{proof}
We will establish a bijection $\Phi$ from the holonomy invariant harmonic
measures to right $G$-invariant measures. For a measure $\mu$ on
transversals of $\hat\Lam$ which is holonomy invariant, we set $\Phi(\mu)$
to simply be the product of  the (holonomy invariant) transverse conditional
measures of $\mu$ with the image of the left Haar measure under the discrete
action of $\Gamma_\L$ on the universal cover of leaves of $\hat\L$ identified
with $G$. Since $G$ is unimodular, these measures are the same as the right
Haar measures on $G$. Hence the global measure $\Phi(\mu)$ is right
$G$-invariant.

For the converse we use the identification of the various definitions of
holonomy invariance given by A. Connes in \cite{Connes:82}. A holonomy
invariant measure can be thought of an equivalence class of objects of the form
$[\omega,\mu]$ where $\omega$ is a volume form on the leaves and $\mu$ is
a measure on $\hat\Lam$ which is $\omega$-invariant.  If $\mu$ is a
$G$-invariant measure on $\hat\Lam$ and $\omega$ is the (left=right) Haar
measure on the leaves, the pair $[\omega,\mu]$ determines a
holonomy-invariant measure on transversals of $\hat\Lam$ (and therefore of
$\Lam$.)
\end{proof}

%%%%%%%%%%%%%%%%%%%%%%%%%%%%%%%%%%%%%%%%%%
%%%%%%%%%%%%%%%%%%%%%%%%%%%%%%%%%%%%%%%%%%
%%%%%%%%%%%%%%

%\bibliographystyle{myalpha}
%\bibliography{allbib}

\def\cprime{$'$} \def\polhk#1{\setbox0=\hbox{#1}{\ooalign{\hidewidth
  \lower1.5ex\hbox{`}\hidewidth\crcr\unhbox0}}}
\providecommand{\bysame}{\leavevmode\hbox to3em{\hrulefill}\thinspace}
\providecommand{\MR}{\relax\ifhmode\unskip\space\fi MR }
\providecommand{\MRhref}[2]{%
  \href{http://www.ams.org/mathscinet-getitem?mr=#1}{#2}
}
\providecommand{\href}[2]{#2}

%%%%%%%%%%%%%%%%%%%%%%%%%%%%%%%%%%%%%%%%%%
%%%%%%%%%%%%%%%%%%%%%%%%%%%%%%%%%%%%%%%%

\end{document}